\documentclass[12pt, reqno]{amsart}
\usepackage{geometry}                
\usepackage{graphicx}
\usepackage{booktabs}
\usepackage{amssymb}
\usepackage{epstopdf}
\usepackage{hyperref}
\makeatletter
\def\@setthanks{\vspace{-\baselineskip}\def\thanks##1{\@par##1\@addpunct.}\thankses}
\makeatother

\newtheorem{theorem}{Theorem}
\newtheorem{lemma}[theorem]{Lemma}
\newtheorem{definition}[theorem]{Definition}
\newtheorem{corollary}[theorem]{Corollary}

\newtheorem{assumption}{A}

\title[On the GMFG equations]
{On the Graphon Mean Field Game Equations: Individual Agent Affine Dynamics and Mean Field Dependent Performance Functions}

\author{
Peter E. Caines}
\author{
Daniel  Ho}
\author{Minyi Huang}
\author {
Jiamin Jian}
\author {
Qingshuo Song}
\thanks{
\noindent 
P. E. Caines is with the
Department of Electrical and Computer Engineering, McGill University, Montreal,  Canada. peterc@cim.mcgill.ca. \\
D. Ho is with the
Department of Mathematics, City
University of Hong Kong, Hong
Kong. madaniel@cityu.edu.hk.\\
M. Huang is with the
School of Mathematics and Statistics, Carleton University, Ottawa, ON,
Canada. mhuang@math.carleton.ca.\\
Q. Song and J. Jian are both with the Department of Mathematical Sciences, Worcester Polytechnic Institute, 
jjian2@wpi.edu,
qsong@wpi.edu.\\
The research of P.E. Caines, D. Ho, and Q. Song were 
supported in part  by the RGC of Hong Kong CityU (11201518).
The work of P. E. Caines was partially supported by AFOSR grant  FA9550-19-1-0138. 
The research work of M. Huang was supported by NSERC.
\\ We acknowledge the valuable comments from anonymous reviewers.
}

\begin{document}
\maketitle
\begin{abstract}
This paper establishes unique solvability of a class of Graphon Mean Field Game equations. 
The special case of a constant graphon yields the result for the Mean Field Game equations.
\end{abstract}

\section{Introduction}
Mean Field Game (MFG) theory establishes  Nash equilibirum conditions for  large populations of asymptotically negligible non-cooperating agents via an analysis of the infinite limit population (Huang, Caines, and Malhame \cite{HCM03, HCM07, HCM06}; Lasry and Lions \cite{LL07}). 
The resulting PDEs (Partial Differential Equations) consist of a backward Hamilton-Jacobi-Bellman (HJB) equation and a forward Fokker-Planck-Kolmogorov (FPK) equation for each generic agent. These equations are linked by the state distribution of a generic agent which is called the mean field of the system.

The basic structure of standard MFG theory assumes a symmetry in the connections of the agents but not necessarily of their dynamics. However,  in the recent studies \cite{CM18, CH19, CH20} asymmetric graph connections in large population games are considered.  Large subpopulations (or clusters) of agents are placed at their particular nodes and communicate with the neighbouring  subpopulations  via the graph edges. The graphs are heterogeneous with the edges having not necessarily identical weights. In the network limit, a graphon gives the communication weights $g(\alpha, \beta)$, see for instance the introductions to each of \cite{CM18, CH19, CH20, GC18} for the Graphon MFG  (GMFG) framework  and 
 \cite{Lov12} for Graphon theory.
Along with \cite{CM18, CH19, CH20}, this paper proposes a new type of MFG PDE system associated to the Graphon Mean Field Game system. Our goal here is to establish the unique solvability of the GMFG equation in an appropriate function space.

The GMFG equations consist of a collection of parameterized Hamilton-Jacobi- Bellman equations, $HJB(\alpha), \alpha \in [0, 1]$, and 
a collection of parameterized Fokker-Planck-Kolmogorov equations, $FPK(\alpha)$ with $\alpha \in [0,1]$. 
The solution of a set of GMFG equations is a parameterized pair $(v, \mu)$, where 
$v[\alpha] = v(t, \alpha, x)$ solves the $HJB(\alpha)$ equation and 
$\mu[\alpha] = \mu(t, \alpha, x)$ solves the $FPK(\alpha)$ equation.
 The coupling of the system PDEs in this paper has the following features (see \cite{CH20} for a more general framework subject to different  hypotheses): 
\begin{itemize}
\item 
$FPK(\alpha)$ depends upon $HJB(\alpha)$ through its first order coefficient $\nabla v$. 
\item $HJB(\alpha)$ depends upon $FPK(\alpha')$  for all $\alpha' \in [0,1]$ through the graphon $g$ acting on 
$\mu[\alpha']$;  this is the major difference from MFG.
\end{itemize}

The GMFG equations with a constant graphon reduce 
to the classical MFG system as a special case, and the original methods to establish solvability of the classical MFG equations are helpful in the present case.
 In 
 \cite{HCM06} and \cite{NC13},  
a Banach fixed point analysis is used depending on a 
contraction argument; this is based on assumptions on the Lipschitz continuity of the functions appearing in the MFG equations and their derivatives, and yields uniquenss as well as existence. 
This approach is used in the parallel study  \cite{CH20} of the solvability of the GMFG  equations. On the other hand, 
\cite{Car13} and \cite{Ryz18} carry out the existence analysis 
using the Schauder fixed point theorem based upon regularity assumptions and then obtain uniqueness via a monotonicity assumption on the running cost.

In this work, similar to the aforementioned analyses, we will establish the existence of solutions via the application of a fixed point theorem. 
Our existence proof adopts Schauder's argument on the fixed point theorem and is more closely relevant to \cite{GPV16}, \cite{Car13}, and \cite{Ryz18} in this sense.
Unlike \cite{GPV16} on the solvability in Sobolev space, our solvability is  to answer the existence in H\"older space along the lines of \cite{Car13} and \cite{Ryz18}.
Nevertheless, different from all aforementioned papers, our proof on the continuity of the gradient of the value function with respect to the coefficient functions relies on probabilistic estimates rather than the theory of viscosity solutions. The main advantage of our approach is that we can conclude the local Lipschitz continuity of the solution map, 
which is stronger than continuity and beneficial to the subsequent analysis of the GMFG.

Having said that, the major difficulty generalizing existence from the MFG case to the GMFG case 
is to obtain the regularity of the solution with respect to the variable $\alpha$, which is essential for 
the existence result by Schauder's fixed point theorem. To be more illustrative, 
for instance, to obtain a uniform first order estimate of
$|\nabla v(t, \alpha, x) - \nabla v(t, \alpha', x)|$ for the solution $v$ of the HJB equation, one has to compare the solutions from two different HJBs parameterized by $\alpha$ and $\alpha'$. 
This leads to a study of the sensitivity with respect to coefficient functions of corresponding PDEs. Therefore, 
the local Lipschitz continuity of the HJB solution map becomes essential for this procedure.

 The paper is organized as follows. Section \ref{s:ps} gives the problem set up.
 Section \ref{s:reg} presents the regularity of parabolic PDE  and applies this to the FPK. Section \ref{s:exist} presents the existence result and Section \ref{s:unique} treats uniqueness. Section \ref{s:conclusion} presents a summary and extensions of the main result. For better clarity, all notations used in this paper have been collected and explained in the Appendix Section \ref{s:appendix}.

\section{Problem setup}\label{s:ps}

Let $\mathbb T^d$ be a d-torus. $\mathcal P_1(\mathbb T^d)$ 
is  the Wasserstein space of probability measures on $\mathbb T^d$ satisfying 
$$\int_{\mathbb T^d} |x| d \mu(x) < \infty$$
endowed with 1-Wasserstein metric  $d_1(\cdot, \cdot)$ defined by
$$d_1(\mu, \nu) = \inf_{\pi \in \Pi(\mu, \nu)} \int_{\mathbb T^d \times \mathbb T^d} |x - y| d\pi(x, y),$$
where $\Pi(\mu, \nu)$ is the collection of all probability measures on $\mathbb T^d \times \mathbb T^d$ with its marginals agreeing with $\mu$ and $\nu$.

We consider the following large system of multi-agent problems.
A generic agent can be identified  
by its state pair $(\alpha, x) \in [0,1]\times \mathbb T^d$, 
where 
$\alpha$ is the cluster index 
and $x$ is a $\mathbb T^d$ valued state. The weights of connections between clusters are given by  a symmetric measurable function $g: [0,1]^2 \mapsto \mathbb R$, which is commonly referred to a graphon \cite{Lov12}.
The population density at the cluster $\alpha$ at time $t$ will be given by 
$\mu(t, \alpha) \in \mathcal P_1(\mathbb T^d)$.

\begin{proof}[Example]
Two examples of graphons are given in the following discussion, 
while the reader is  referred to \cite{Lov12} for the fundamental theory of this subject. 
A uniform graphon which corresponds to the limit of a sequence of Erd\"os-R\'enyi graphs with parameter $p$, $0 \leq p \leq 1$, is given by
\begin{equation}
\label{eq:graphon01}
g(\alpha, \alpha') = p, \ \forall \alpha, \alpha' \in [0,1]
\end{equation}
and the uniform attachment graph limit  has the graphon
\begin{equation}
\label{eq:graphon02}
g(\alpha, \alpha') = 1 - \max\{\alpha, \alpha'\}, \ \forall \alpha, \alpha' \in [0,1].
\end{equation}
\end{proof}

A running cost incurred to the generic agent of $(\alpha, x)$ with a feedback control exertion ${\bf a}: [0, T]\times [0,1] \times \mathbb T^d \mapsto  \mathbb R^d$ at time $t$ is given by
\begin{equation}
\label{eq:cost02}
\ell(\mu, g,  {\bf a}, t, \alpha, x)= \frac 1 2 |{\bf a}(t, \alpha, x)|^2 + \ell_1(\mu, g, t, \alpha, x)
\end{equation}
for some given function $\ell_1(\cdot, \cdot, \cdot, \cdot, \cdot)$. The following cost can be considered  as an example for $\ell_1$
\begin{equation}
\label{eq:cost1}
\ell_1(\mu, g, t, \alpha, x) = \int_0^1 \int_{\mathbb T^d} \ell_2(x, y) \mu(t, \alpha', dy)  g(\alpha, \alpha') \ d \alpha'
\end{equation}
for some 
$\ell_2: \mathbb T^d \times \mathbb T^d
\mapsto \mathbb R$.


Let $b: [0, T]\times [0,1] \times \mathbb T^d \mapsto \mathbb R$ and
$m_0: [0, 1]\times \mathbb T^d \mapsto \mathbb R^+$ be two given smooth enough functions. 
By $\nabla b$, we denote the gradient of $b$ on the domain $\mathbb T^d$, which is mapping  
$[0, T]\times [0,1] \times \mathbb T^d \mapsto \mathbb R^d$.
Finding a solution of the GMFG equations consists of solving for the unknown triples $(v, {\bf a}^*, \mu)$:
\begin{itemize}
\item
 the value function $v:    [0, T] \times [0, 1] \times \mathbb T^d \mapsto \mathbb R$, 
 \item
optimal control ${\bf a}^*:  [0, T] \times [0, 1] \times \mathbb T^d \mapsto \mathbb R^d$, 
\item and
the density $\mu:  [0, T] \times [0, 1] \times \mathbb T^d \mapsto \mathbb R^+$,
\end{itemize}
satisfying the $\alpha$ parameterized family
\begin{equation}
\label{eq:gmfg01}
\left\{
\begin{array}
{ll}
\displaystyle
\partial_t v  + 
(\nabla b +  {\bf a}^*)  \cdot
\nabla v
+ \frac 1 2 \Delta v 
+ \ell(\mu, g,  {\bf a}^*  ) = 0 \\
\displaystyle
{\bf a}^*(t, \alpha, x) = 
\arg \min_{a \in \mathbb R^d} 
\{  a \cdot \nabla v (t, \alpha, x) + \frac 1 2 |a|^2  \} \\
\displaystyle
\partial_t \mu = - \text{div}_x ( (\nabla b+  {\bf a}^*)  \mu ) + \frac 1 2 \Delta \mu \\
\displaystyle
v(T, \alpha, x) = 0, \quad \mu(0, \alpha,  x) = m_0(\alpha, x).
\end{array}
\right.
\end{equation}
In the first and third equation of \eqref{eq:gmfg01}, each term is a function of $(t, \alpha, x)$ without further specification. In particular, 
the $\ell(\mu, g,  {\bf a}^*  ) $ shall be understood as a mapping
$$(t, \alpha, x) \mapsto \ell(\mu, g,  {\bf a}^*  )(t, \alpha, x) :=  \ell(\mu, g,  {\bf a}^*, t, \alpha, x).$$

Our goal in this paper is to establish existence, uniqueness for the solution of \eqref{eq:gmfg01} in an appropriate solution space. 
We close this section with a brief illustration of the probabilistic formulation on the GMFG for the motivational purpose.
A generic player in GMFG is identified by a pair $(\alpha, x)\in [0,1] \times \mathbb T^d$, where $\alpha$ is geographical information and $x$ is a state.
The population density at index $\alpha$ at time $t$ is denoted by $\mu(t, \alpha) \in \mathcal P_1(\mathbb T^d)$
and the relation between two generic players in $\alpha$ and $\alpha'$ is given by a graphon $g(\alpha, \alpha')$.  
Given a population density $(t, \alpha) \mapsto \mu(t, \alpha)$ and a graphon $(\alpha, \alpha') \mapsto g(\alpha, \alpha')$, 
a generic player exerts its optimal strategy of the following stochastic control problem described below.
State evolution of the generic player at $\alpha$ follows a controlled stochastic differential equation:
\begin{equation}
\label{eq:Xt}
X_t^\alpha = X_0^\alpha + \int_0^t (\nabla b(s, \alpha, X_s^\alpha) + {\bf a}(s, \alpha,  X_s^\alpha)) ds + W_t^\alpha, 
\end{equation}
where the drift is formed by a control process $\bf a$ and a conservative vector field $\nabla b$,  
$W^\alpha$ is a Brownian motion in a filtered probability space independent to $W^\beta$ for any $\beta \neq \alpha$,
and $X_0^\alpha$ is an initial random variable with a given distribution $m_0(\alpha)$.
In the above, the left hand side is understood as the coset of $\mathbb Z^d$ that contains the right hand side by a mapping $\pi(x) = x + \mathbb Z^d$.
We use $X^\alpha [{\bf a}]$ to denote the process with the dependence on ${\bf a}$.
The objective of the generic player at $\alpha$ with a given population density flow $\mu$ is to minimize the total cost incurred during $[0, T]$ of the form
$$J^\alpha( {\bf a}, \mu) = \mathbb E \Big[ \int_0^T \ell(\mu, g, {\bf a}, t, \alpha, X_t^\alpha [{\bf a}]) dt \Big]$$
over a reasonably rich enough control space of ${\bf a}$. Note that the optimal strategy ${\bf a}^*$ depends on $\mu$. 
Given an initial distribution $m_0$, the goal of the GMFG is  to find the Nash equilibrium $\mu^*$ and the corresponding ${\bf a}^*$, i.e.
the pair $(\mu^*, {\bf a}^*)$ satisfies
$$J^\alpha ({\bf a}^*, \mu^*) \le J^\alpha ({\bf a}, \mu^*), \forall {\bf a} \hbox{ and } \mu^*(t, \alpha) \sim X^\alpha_t[{\bf a}^*], \forall (t, \alpha).$$
Indeed, the above formulation poses a class of mean field game problems indexed by $\alpha\in [0,1]$ and couplings between mean field games are imposed by the running cost $\ell$ via graphon $g$. 
For more detailed discussion and various applications are referred to  \cite{CM18, CH19, CH20}.

\section{Some regularity results}\label{s:reg}
We are going to present sensitivity results of the parabolic PDE and FPK equations with respect to their coefficients separately, which eventually
serve for the proof of fixed point theorem as key elements. 
Throughout the paper, we will use $\Psi(\cdot)$ in various places as a generic positive function increasing with respect to its variables.
Morevoer, all function spaces and relevant norms are sorted out in Section \ref{s:appendix}.
\subsection{Parabolic equations}
Consider the equation
\begin{equation}
\label{eq:ppde01}
\left\{
\begin{array}
{ll}
\partial_t u = \frac 1 2 \Delta u - c u + f, \hbox{ on } (0, T)\times \mathbb T^d
\\
u(0, x) = 0, \hbox{ on } x\in \mathbb T^d.
\end{array}
\right.
\end{equation}
We will denote the solution map by $u = u[ c, f]$ whenever it is necessary to emphasize its dependence on the coefficient functions.
\subsubsection{Preliminaries on solvability}
If the coefficients $ c$ and $f$ are H\"older in both variables $(t, x)$, then
there exists a unique classical solution. Recall that $\Psi(\cdot)$ is a generic function mentioned in the first paragraph of Section 3.
\begin{lemma}
\label{l:ppde01}
If $c, f\in C^{\delta/2, \delta} ([0, T] \times \mathbb T^d)$ holds for some $\delta\in (0,1)$, then there exists unique solution
$u \in C^{1+\delta/2, 2+\delta}  ([0, T] \times \mathbb T^d)$ of  \eqref{eq:ppde01} satisfying
$$|u|_{1+\delta/2, 2+\delta} \le \Psi(|c|_{\delta/2, \delta}, |f|_{\delta/2, \delta}).$$
Moreover, 
$v(t, x) := u(T-t, x)$ has a probabilistic representation $v[c, f]$ of the form
\begin{equation}
\label{eq:G02}
v(t, x) = v[c, f] (t, x) := 
\mathbb E \Big[ \int_t^{T} \exp\{- \int_t^{s} c(r, X^{t,x}(r)) dr\} f(s, X^{t,x}(s) )ds\Big], 
\end{equation}
where 
\begin{equation}
\label{eq:X01}
X^{t, x} (s)= x  + W(s) -W(t)
\end{equation}
for some Brownian motion $W$.
\end{lemma}

\begin{proof}
The solvability and its H\"older estimate is from  Theorem 8.7.2 and 
Theorem 8.7.3 of \cite{Kry96}, Theorem IV.5.1 of \cite{LSU67}.
The probabilistic representation $v[c, f]$ is from Feynman-Kac formula, see \cite{FS06}.
\end{proof}

In the above, 
we remark that, \eqref{eq:X01} reads by $X^{t, x} (s)= \pi(x  + W(s) -W(t))$, where $\pi$ is the generic mapping $\mathbb R^d \mapsto \mathbb R^d/\mathbb Z^d$.
Later we also need to use the following definition of weak solution, see \cite{Eva98}.
\begin{definition} \label{d:wkpde}
A function $u \in L^2 ([0, T], H^1(\mathbb T^d))$ 
is a weak solution of \eqref{eq:ppde01} if 
$u$ satisfies
\begin{equation}
\label{eq:ppde03}
\left\{
\begin{array}
{ll}
\int_{\mathbb T^d} \phi (- \partial_t u 
 - c u + f) dx
= \frac 1 2  \int_{\mathbb T^d}  \nabla \phi \cdot \nabla u dx, \ \forall \phi \in 
H^1(\mathbb T^d)
\\
u(0, x) = 0, \hbox{ on } x\in \mathbb T^d.
\end{array}
\right.
\end{equation}
\end{definition}

We have the following uniqueness with the same assumptions as in Lemma \ref{l:ppde01}.
\begin{lemma}
\label{l:ppde03}
If $c, f\in C^{\delta/2, \delta} ([0, T] \times \mathbb T^d)$ 
holds for some $\delta\in (0,1)$, then there exists unique 
weak solution of\eqref{eq:ppde01} in $L^2 ([0, T], H^1(\mathbb T^d))$.
\end{lemma}
\begin{proof}
By Lemma \ref{l:ppde01}, there exists a classical solution $u$. Together with
the compactness of the domain, it yields $u \in L^2  ([0, T], H^1(\mathbb T^d))$.
By Theorem 7.4 of \cite{Eva98}, uniqueness in $L^2 ([0, T], H^1(\mathbb T^d))$ holds if $c \in L^\infty$ and $f\in L^2$, and this is valid, since all coefficients are 
continuous on the compact domain.
\end{proof}

\subsubsection{First order regularity and sensitivity of the solution map}
Although Lemma \ref{l:ppde01} has an estimation on $|u|_{1, 2}$, 
it is controlled by an upper bound relevant to the H\"older norm of coefficients in the $t$ variable, 
which is not desirable, see Section \ref{s:remarks} for further remarks. Next, we will develop an upper bound 
independent of  $t$-H\"older norm of the coefficients.
To proceed, we define a linear operator
\begin{equation}
\label{eq:L01}
L u = \partial_t u - \frac 1 2\Delta u.
\end{equation}
The first result is on an estimate of  $|u|_0 = \sup_{[0, T] \times {\mathbb T}^d} |u(t, x)|$. 
\begin{lemma}
\label{l:reg01}
If $c, f\in C^{\delta/2, \delta} ([0, T] \times {\mathbb T}^d)$, then $u$ of \eqref{eq:ppde01} satisfies  $|u|_0 \le e^{|c|_0T} |f|_0 T$.
\end{lemma}
\begin{proof}
If $c = 0$,  then with $u_1 =  |f|_0 t$,
$$L u_1 - f = |f|_0 - f\ge 0.$$
If $c\neq 0$, then with $u_2 = \frac{|f|_0 (e^{|c|_0 t} -1)}{|c|_0}$,
$$
\begin{array}
{ll}
(L +c) u_2 
& \displaystyle = |f|_0 e^{|c|_0t} \Big(1 + \frac{c}{|c|_0}\Big) - \frac{c}{|c|_0} |f|_{0}
\\
& \displaystyle
=
 |f|_0 (e^{|c|_0t}-1) \Big(1 + \frac{c}{|c|_0}\Big) + |f|_{0}
 \\ & \displaystyle
\ge f.
\end{array}
$$
Note that both $u_1$ and $u_2$ are no greater than $e^{|c|_0T} |f|_0 T$, and finally the comparison principle yields the result.
\end{proof}

Next we will have the first order estimate independent to the H\"older norm in $t$ of the coefficients. It also gives sensitivity of the solution map with respect to the coefficients.
\begin{lemma}
\label{l:lip01}
Let  $c, f$ be in $C^{\delta,1}([0, T]\times \mathbb T^d)$ for some  $\delta \in (0,1)$. 
Then the solution $u$ of \eqref{eq:ppde01} belongs to 
$C^{1, 2} ([0, T]\times \mathbb T^d)$ with 
$$|u|_{0,1} \le \Psi(|c|_{0,1} + |f|_{0,1}).$$
Furthermore, the solution map $u = u[c, f]$ satisfies
$$|u[c_1, f_1]  - u[c_2, f_2]  |_0 \le \Psi(K)  (|c_1- c_2|_0 + |f_1-f_2|_0)
$$
for
$ K:=|c_1|_{0}+ |c_2|_{0}+ |f_1|_{0}+ |f_2|_{0}.$
\end{lemma}
\begin{proof}
$u$ of \eqref{eq:ppde01} can be written by $u (t, x) = v[c, f](T - t, x)$ with its probabilistic representation of \eqref{eq:G02}.
By setting $X^i: = X^{t, x_i}$ of \eqref{eq:X01}, we have
$$X^1_s - X^2_s = x_1 - x_2, \ \forall s\ge t.$$  
If we define
$$\Lambda_s^i = e^{-\int_t^s c(r, X^i(r)) dr},$$
then
$$v[c, f] (t, x_i) = \mathbb E \Big[ \int_t^T \Lambda^i_s f(s, X^i(s)) ds \Big ].$$
We first note that, by mean value theorem 
$$
\Big|\int_t^s c(r, X^1(r) )dr - \int_t^s c(r, X^2(r)) dr \Big| \le T |c|_{0,1}  |x_1 - x_2|. 
$$
Once again by mean value theorem and the fact of $|-\int_t^s c(r, X^i(r)) dr| \le T |c|_0$, we obtain 
$$
\begin{array}
{ll}
|\Lambda^1_s - \Lambda^2_s|  & \le  e^{T |c|_0} \Big|\int_t^s c(r, X^1(r)) dr - \int_t^s c(r, X^2(r) )dr \Big| \\
& \le  \Psi(|c|_{0,1})  |x_1 - x_2|
\end{array}
$$
with probability one for $\Psi = T |c|_{0,1} e^{T|c|_0}$.
Therefore, we have
$$
\begin{array}
{ll}
 \displaystyle
|v[ c, f] (t, x_1) - v[ c, f] (t, x_2) |
\le \mathbb E \int_t^T |  \Lambda^1_s f(s, X^1(s))  -  \Lambda^2_s f(s, X^2(s)) | ds
\\ \displaystyle \quad \quad \quad
\le \mathbb E \int_t^T ( |\Lambda^1|_0 |f(s, X^1(s)) -  f(s, X^2(s))| + 
|f|_0 | \Lambda^1_s -  \Lambda^2_s| ) ds
\\ \displaystyle \quad \quad \quad
\le \mathbb E \int_t^T \Psi(|c|_0) |\nabla f|_0 |X^1(s) -  X^2(s)| ds 
+ 
\mathbb E \int_t^T  |f|_0 | \Lambda^1_s -  \Lambda^2_s| ds
\\ \displaystyle \quad \quad \quad
\le 
T\Psi(|c|_0) |\nabla f|_0   |x_1 -x_2|
+ 
T |f|_0 | \Psi(|c|_{0,1})  |x_1 - x_2|
\\ \displaystyle \quad \quad \quad
\le \Psi(|c|_{0,1}+|f|_{0,1}) |x_1 - x_2|
\end{array}
$$
This implies that $|\nabla v|_0 \le  \Psi(|c|_{0,1}+|f|_{0,1})$. Together with Lemma \ref{l:reg01}, we conclude that 
\begin{equation}
\label{eq:lem4-1}
|u|_{0,1} \le  \Psi(|c|_{0,1}+|f|_{0,1}).
\end{equation}

Next, we estimate $|u[ c, f_1] - u[ c, f_2] |_0$. For any $(t,x)$, we set
$\Lambda_s = e^{-\int_t^s c(r, X(r)) dr},$ and note that
$$
\begin{array}
{ll}
|v[ c, f_1] (t,x)- v[ c, f_2] (t,x) | & \le \displaystyle
\mathbb E \int_t^T |\Lambda_s f_1(s, X_s) - \Lambda_s f_2(s, X_s)| ds 
\\ 
& \displaystyle \le |f_1 - f_2|_0 \mathbb E \int_t^T |\Lambda_s|  ds 
\\ 
& \displaystyle \le  T e^{T |c|_0}  |f_1 - f_2|_0.
\end{array}
$$
This concludes that 
\begin{equation}
\label{eq:lem4-2}
|u[ c, f_1] - u[ c, f_2] |_0 \le \Psi(|c|_0) |f_1 -f_2|_0.
\end{equation}
In the following, we estimate  $|u[ c_1, f] - u[ c_2, f] |_0$. By setting 
$\Lambda^i_s = e^{-\int_t^s c_i(r, X(r)) dr},$ we have
$$|\Lambda^1_s - \Lambda^2_s| 
\le e^{T(|c_1|_0 + |c_2|_0)} \int_t^s |c_1(r, X_r) - c_2(r, X_r)|dr
\le e^{T(|c_1|_0 + |c_2|_0)}  |c_1 - c_2|_0 T$$
with probability one. Therefore,
$$
\begin{array}
{ll}
|v[c_1, f] (t,x)- v[c_2, f] (t,x) | & \le \displaystyle
\mathbb E \int_t^T |\Lambda^1_s f(s, X_s) - \Lambda^2_s f(s, X_s)| ds 
\\ 
& \displaystyle \le |f|_0 \mathbb E \int_t^T |\Lambda^1_s - \Lambda^2_s|  ds 
\\ 
& \displaystyle \le  T^2 e^{T(|c_1|_0 + |c_2|_0)} |f|_0  |c_1 - c_2|_0.
\end{array}
$$
This implies that 
\begin{equation}
\label{eq:lem4-3}
|u[ c_1, f] - u[ c_2, f] |_0 \le \Psi(|c_1|_0 + |c_2|_0 + |f|_0) |c_1-c_2|_0.
\end{equation}
The conclusion yields from \eqref{eq:lem4-1},  \eqref{eq:lem4-2},  \eqref{eq:lem4-3}.
\end{proof}

\subsubsection{Second order regularity and first order sensitivity}

Next, we will see that under better regularity of $c$ and $f$ in $x$, 
we can improve regularity and sensitivity. 
Formally, if $u$ of \eqref{eq:ppde01} is smooth enough, one can take derivatives of the equation to conclude that $\bar u_j = \partial_j u$ 
is the solution of the following equation depending on  $c, f$ and $u$ 
of \eqref{eq:ppde01}.
\begin{equation}
\label{eq:ppde02}
\left\{
\begin{array}
{ll}
\partial_t \bar u_j = \frac 1 2 \Delta \bar u_j 
- c \bar u_j 
- u \partial_j c  + \partial_j f, 
\hbox{ on } (0, T)\times {\mathbb T}^d
\\
\bar u_j(0, x) = 0, \hbox{ on } x\in {\mathbb T}^d
\end{array}
\right.
\end{equation}
However, \eqref{eq:ppde02} is valid only if $u\in C^{1,3}$ is given a priori. 

\begin{lemma}
\label{l:ubar01}
If $ c, f\in C^{\delta,2}([0, T] \times {\mathbb T^d})$ for some $\delta \in (0,1)$, then
the solution $u$ of \eqref{eq:ppde01} is in $C^{1,3}([0, T] \times {\mathbb T^d})$ and 
$\bar u_j = \partial_j u$  is the unique solution of \eqref{eq:ppde02}.
\end{lemma}
\begin{proof}
By Lemma \ref{l:ppde03}, $u$ satisfies, for any $\phi\in H^2(\mathbb T^d)$, 
$$
\int_{\mathbb T^d} \phi (- \partial_t u  - c u + f) dx
= \frac 1 2  \int_{\mathbb T^d}  \nabla \phi \cdot \nabla u dx.
$$
Now, if we replace the test function $\phi$ by $\partial_i \phi$ in the above variational form, then
we have
$$
\int_{\mathbb T^d} \partial_i \phi (- \partial_t u  - c u + f) dx
= \frac 1 2  \int_{\mathbb T^d}  \nabla \partial_i \phi \cdot \nabla u dx.$$
Using integration by parts, we can show that $\bar u_j$ solves the variational form of \eqref{eq:ppde02} for any $\phi \in H^2(\mathbb T^d)$. Since $H^2(\mathbb T^d)$ is a dense subset in $H^1(\mathbb T^d)$,  $\bar u_j$ is indeed a unique weak solution of \eqref{eq:ppde02}. 

Lastly, since the  $ \nabla c,\nabla f \in C^{\delta,1}([0, T] \times {\mathbb T^d})$, we conclude that $\bar u_j$ is indeed a classical solution from Lemma \ref{l:ppde01}. This also implies that $u\in C^{1,3}([0, T] \times {\mathbb T^d})$.
\end{proof}

\begin{lemma}
\label{l:reg03}
Let $c, f\in C^{\delta,2}([0, T]\times \mathbb T^d)$. 
Then the solution $u$ of \eqref{eq:ppde01} belongs to 
$C^{1, 3} ([0, T]\times \mathbb T^d)$ with 
$$|u[c,f]|_{0,2} \le \Psi(|c|_{0,2}+ |f|_{0,2}).$$
Furthermore, the solution map $u = u[c, f]$ of \eqref{eq:ppde01} satisfies
$$| u[ c_1, f_1]  -  u[ c_2, f_2]  |_{0, 1} 
\le \Psi(K)  ( |c_1- c_2|_{0,1} + |f_1-f_2|_{0,1})
$$ 
for  
$$K := |c_1|_{0,1}+ |c_2|_{0,1}+ |f_1|_{0,1}+ |f_2|_{0,1}.$$ 
\end{lemma}
\begin{proof}
By Lemma \ref{l:ubar01},  
$\bar u_j = \partial_j u$ is the classical solution of \eqref{eq:ppde02}, which satisfies 
$$\bar u_j = u[c, \bar f],$$
where
$$\bar f = - u \partial_j c  + \partial_j f.$$
Applying Lemma \ref{l:lip01}, we have
$|\bar u_j|_{0,1} < \Psi( |c|_{0,1} + |\bar f|_{0,1} )$. 
Note that, 
$|\bar f|_{0,1}$ is controlled by
$|u|_{0,1}+ | \partial_j c|_{0,1} + |\partial_j f|_{0,1}$, which implies that 
$|\bar f|_{0,1} \le \Psi(|c|_{0,2} + |f|_{0,2})$ due to Lemma \ref{l:lip01}.
Hence, we conclude that $|u[c,f]|_{0,2} \le \Psi(|c|_{0,2}+ |f|_{0,2}).$

At last, applying Lemma \ref{l:lip01} on $ u[c, \bar f]$ again, we have 
$$|u[ c_1, \bar f_1]  - u[c_2, \bar f_2]  |_0 \le \Psi(K)  (|c_1- c_2|_0 + |\bar f_1- \bar f_2|_0)
$$
for $K = |c_1|_0 + |\bar f_1|_0 + |c_2|_0 + |\bar f_2|_0$, which  similarly concludes the desired result.
\end{proof}

\subsubsection{Summary on regularity and sensitivity}

Now we may summarize and generalize the results above to a PDE with non-zero initial conditions.
Consider equation

\begin{equation}
\label{eq:ppde04}
\left\{
\begin{array}
{ll}
\partial_t u =  \frac 1 2 \Delta u - c u + f, \hbox{ on } (0, T)\times \mathbb T^d
\\
u(0, x) = \psi(x), \hbox{ on } x\in \mathbb T^d.
\end{array}
\right.
\end{equation}

To proceed, we recall the following notations:
\begin{itemize}
\item
$C^{\delta, n}_{0, n'}$ be the space of all functions $f\in C^{\delta, n}([0,  T]\times \mathbb T^d)$ with the topology induced by 
the norm $|\cdot|_{0, n'}$. 
\item  $C^{1,3}_{0,1}([0, T]\times \mathbb T^d)$ is the space of all $u\in C^{1,3}([0, T]\times \mathbb T^d)$ topologized by $|\cdot|_{0,1}$.
\end{itemize}
For more details, we refer it to Section \ref{s:appendix}.

\begin{theorem}
\label{t:ppde01}
The solution map $u: [c, f, \psi] \mapsto u[c, f, \psi]$ given by \eqref{eq:ppde04} 
is a locally Lipschitz continuous map
$$C^{\delta,2}_{0,1} \times C^{\delta,2}_{0,1} \times C^{4}_3
\mapsto 
C^{1,3}_{0,1}.
$$
\end{theorem}
\begin{proof}
It is enough to show that
$$
\begin{array}
{ll}
| u[c_1, f_1, \psi_1]  -  u[c_2, f_2, \psi_2]  |_{0, 1} 
\le
\\ \hspace{1in} \Psi(K)  ( |c_1- c_2|_{0,1} + |f_1-f_2|_{0,1} + |\psi_1 -\psi_2|_{3})
\end{array}
$$
for $K =  |c_1|_{0,1} + |c_2|_{0,1} + |f_1|_{0,1} + |f_2|_{0,1} + |\psi_1|_3 + |\psi_2|_3$.
Indeed,
setting $ \tilde u (t, x) =  u(t, x) - \psi(x)$, we have 
$$\tilde u = u [c,  f  + \frac 1 2 \Delta \psi - c \psi, 0]$$
for the solution map $u[\cdot, \cdot, \cdot]$ defined via \eqref{eq:ppde04},
and observe that the desired result  is a consequence of Lemma \ref{l:reg03}.
\end{proof}

Note that the local Lipschitz continuity of Theorem \ref{t:ppde01} automatically yields its local boundedness, i.e
\begin{equation}
\label{eq:reg01}
| u[c, f, \psi]  |_{0, 1} 
\le
\Psi ( |c|_{0,1} +  |f|_{0,1} + |\psi|_{3})
\end{equation}
for some positive increasing function $\Psi$.
The following Harnack  type inequality will be useful.
\begin{corollary}
\label{c:har01}
If $f \equiv 0$, $\psi = e^b$ for some  $c, b \in C^{\delta,2} ([0, T] \times {\mathbb T^d})$, then the solution $u$ of \eqref{eq:ppde04} satisfies the inequality 
$$ e^{- (|b|_0 + |c|_0 T)} < u(t, x) <  e^{|b|_0 + |c|_0T}, \ \forall (t, x) \in [0, T] \times {\mathbb T^d}.$$
\end{corollary}
\begin{proof}
The inequalities follow  from the representation for $v(t, x) = u(T-t, x)$ in the form of
$$
v(t, x) = \mathbb E \Big[ 
\exp\{ - \int_t^T c(r, X^{t,x}(r)) dr \} \psi(X^{t,x}(T))
\Big], $$
where $X$ is given by  \eqref{eq:X01}.
\end{proof}
\subsection{The FPK equation}

We study  the weak solution of 
FPK equation on $[0,T)\times {\mathbb T^d}$:
\begin{equation}
\label{eq:fpk02}
\left\{
\begin{array}
{ll}
\displaystyle
\partial_t \nu(t,  x) = - \text{div}_x ( b (t,  x)  \nu (t, x)) + \frac 1 2 \Delta \nu  (t,  x) \\
\displaystyle
\nu ( 0, x) = m_0( x).
\end{array}
\right.
\end{equation}
We adopt  the conventional notation of
$$\langle m, \psi\rangle := \int_{\mathbb T^d} \psi(x) m(dx)$$
for any $m\in \mathcal P_1(\mathbb T^d)$ and $\psi: \mathbb T^d\mapsto \mathbb R$ whenever it is well defined.
\begin{definition}
$\nu$ is said to be a weak solution of FPK \eqref{eq:fpk02}, if it satisfies, for any $\phi\in C^\infty_c([0, T]\times {\mathbb T^d})$ $$\langle m_0, \phi(0, x) \rangle + \int_0^T \langle \nu_t,  (\partial_t +  \mathcal L) \phi \rangle dt = 0,$$
where 
$$\mathcal L =  b \cdot \nabla + \frac 1 2 \Delta .$$
\end{definition}
We denote the solution map of \eqref{eq:fpk02} by $\nu = \nu[b, m_0]$.
We recall that $C([0, T], \mathcal P_1({\mathbb T^d}))$ is 
the space of all continuous mappings $\nu: [0, T] \mapsto \mathcal P_1({\mathbb T^d})$ with a metric given by
$$dist(\nu_1, \nu_2) = \sup_t d_1(\nu_1(t), \nu_2(t)),$$
where $d_1$ is $1$-Wasserstein metric for $\mathcal P_1$.
\begin{theorem}
\label{t:fpk01}
Let  $m_0 \in \mathcal P_1(\mathbb T^d)$.
Then the solution map $b \mapsto \nu[b, m_0]$ of \eqref{eq:fpk02} is a locally Lipschitz continuous mapping from 
$C ([0,T] \times {\mathbb T^d})$ to $C([0, T], \mathcal P_1({\mathbb T^d}))$. In particular,  
if $|b_1|_{0} + |b_2|_{0} < K$ then 
$$\sup_t d_1(\nu_1(t), \nu_2(t)) \le \Psi(K) |b_1 - b_2|_0.$$
 Moreover, $\nu = \nu[b, m_0]$  satisfies, 
\begin{equation}
\label{eq:nu01}
d_1(\nu(t), \nu(s)) \le (1+ \sqrt T |b|_0 ) |t -s|^{1/2},
\end{equation}
\begin{equation}
\label{eq:nu02}
\sup_t \int_{{\mathbb T^d}} |x| \nu(t, dx) \le
 \int_{{\mathbb T^d}} |x| m_0(dx) + |b|_0 T + \sqrt T.
\end{equation} 
\end{theorem}
\begin{proof}
If $|b|_{0} <\infty$ and $m_0 \in \mathcal P_1$, then 
$$X(t) = X(0) + \int_0^t b(s, X_s) ds + W(t), \ X(0) \sim m_0$$
has a unique  solution. 
An application of 
It\^{o}'s formula and the definition of the weak solution 
verifies that  $\nu(t) = Law (X(t))$ is the weak solution of \eqref{eq:fpk02}, see \cite{Car13}.
\eqref{eq:nu01} also follows from \cite{Car13}.

Next, \eqref{eq:nu02} follows from
$$\sup_t \mathbb E |X(t)| \le \mathbb E|X(0)| + |b|_0 T + \sqrt T.$$

Let's assume 
$|b_1|_{0} + |b_2|_{0} < K$ and $\nu_1$ and $\nu_2$ are corresponding solutions of \eqref{eq:fpk02}.
We denote by $X_1$ and $X_2$ the solutions of the SDE above. Note that 
$$\begin{array}
{ll}
\mathbb E |X_1(t) - X_2(t)| & \le \mathbb E \int_0^t |b_1(s, X_1(s))  - b_2(s, X_2(s))| ds
\\ \displaystyle
& \le |b_1 - b_2|_0 T + 
K \int_0^t \mathbb E [ |X_1(s) - X_2(s)| ]ds.
\end{array}
$$
So, we can use the Gronwall's inequality to have
$$\mathbb E |X_1(t) - X_2(t)| \le |b_1 - b_2|_0  T e^{KT}.$$
Therefore, we can have local Lipschitz of $b \mapsto \nu [b, m_0]$ from
$$d_1(\nu_1(t), \nu_2(t)) \le \mathbb E |X_1(t) - X_2(t)| \le  |b_1 - b_2|_0 T  e^{KT}.$$
\end{proof}

\section{Existence} \label{s:exist}
We now return to the GMFG scheme.
First observe that, by using the cost of the form  \eqref{eq:cost02}, 
the triple $(v, {\bf a}^*, \mu)$ is the solution of \eqref{eq:gmfg01} 
if and only if
the pair
$(\tilde v := v - b, \mu)$ is the solution of HJB equation
\begin{equation}
\label{eq:gmfg02a1}
\left\{
\begin{array}
{ll}
\displaystyle
\partial_t \tilde v 
- \frac 1 2  |\nabla \tilde  v|^2
+ \frac 1 2 \Delta \tilde v
+ \tilde \ell_1(\mu, g) = 0 \\
\displaystyle
\tilde v(T, \alpha, x) = -b(T, \alpha, x)
\end{array}
\right.
\end{equation}
coupled with
FPK equation
\begin{equation}
\label{eq:gmfg02b1}
\left\{
\begin{array}
{ll}
\displaystyle
\partial_t \mu = 
 \text{div}_x (   \mu \nabla \tilde  v ) + \frac 1 2 \Delta \mu  \\
\displaystyle
\mu ( 0, \alpha, x) = m_0(\alpha, x),
\end{array}
\right.
\end{equation}
where $\tilde \ell_1$ is
\begin{equation}
\label{eq:til-ell}
\tilde \ell_1 (t, \alpha, x) = \ell_1(t, \alpha, x) + (\partial_t b + \frac 1 2 |\nabla b|^2 + \frac 1 2 \Delta b)(t, \alpha, x).
\end{equation}
Next, we outline our approach to the existence as follows. We define an operator 
$$\nu = \Phi(\mu) =  \Phi_2 \circ \Phi_1 (\mu),$$
where
\begin{enumerate}
\item
$\nabla \tilde v = \Phi_1(\mu)$, where $\tilde v$ 
 is the solution of \eqref{eq:gmfg02a} with a given $\mu$:
 \begin{equation}
\label{eq:gmfg02a}
\left\{
\begin{array}
{ll}
\displaystyle
\partial_t \tilde v 
- \frac 1 2  |\nabla \tilde  v|^2
+ \frac 1 2 \Delta \tilde v
+ \tilde \ell_1(\mu, g) = 0 \\
\displaystyle
\tilde v(T, \alpha, x) = -b(T, \alpha, x)
\end{array}
\right.\end{equation}
\item 
$\nu = \Phi_2(\bar v)$ be the function solving \eqref{eq:gmfg02b} with a given $\bar v$:
\begin{equation}
\label{eq:gmfg02b}
\left\{
\begin{array}
{ll}
\displaystyle
\partial_t \nu = 
 \text{div}_x ( \bar v \nu ) + \frac 1 2 \Delta \nu  \\
\displaystyle
\nu ( 0, \alpha, x) = m_0(\alpha, x).
\end{array}
\right.
\end{equation}
\end{enumerate}
The existence of the solution for the GMFG can be accomplished by Schauder's fixed point theorem in an appropriate space to be mentioned below.

To proceed, we recall that $d_1$ is the Wasserstein metric on $\mathcal P_1(\mathbb T^d)$.
We define the space $S^{1/2}$ as the collection of $\mu: [0, T]\times [0, 1] \mapsto \mathcal P_1({\mathbb T^d})$ such that
$$|\mu|_{1/2} = |\mu|_0 + [\mu]_{1/2} <\infty,$$
where
$$|\mu|_0 = \sup_{t, \alpha} \int_{{\mathbb T^d}} |x| \mu(t, \alpha, dx)$$
and
$$[\mu]_{1/2} = \sup_{t_1\neq t_2, \alpha} 
\frac{d_1(\mu(t_1, \alpha), \mu(t_2, \alpha))}{|t_1 - t_2|^{1/2}}.$$
Note that, $S^{1/2}$ is metrizable by
\begin{equation}
\label{eq:rho01}
\rho (\mu_1, \mu_2) = \sup_{t, \alpha} d_1(\mu_1(t, \alpha), \mu_2(t, \alpha)),
\end{equation}
and we denote the space $S^{1/2}$ by $(S^{1/2}, \rho)$ whenever we need to emphasize its underlying metric.
Note that $B_r:= \{\mu\in S^{1/2}: |\mu|_{1/2} \le r\}$ is a closed convex compact subset of $(S^{1/2}, \rho)$ by 
generalized version of Arzelà–Ascoli theorem, see P232 of \cite{Kel17}.

It is often useful by the duality representation of Wasserstein metric to write
\begin{equation}
\label{eq:rho02}
\rho (\mu_1, \mu_2) = \sup_{t, \alpha,  Lip(f) \le 1}  
\int_{\mathbb T^d} f(x) d (\mu_1(t, \alpha) - \mu_2(t, \alpha))(x)
\end{equation}
where $Lip(f)$ is the Lipschitz constant of the function $f$.
Similarly,  if $\mu \in B_r$ and $f\in C^1$, then 
\begin{equation}
\label{eq:mu01}
\begin{array}
{ll}
\displaystyle
\int_{\mathbb T^d} f(y) d(\mu(t_1, \alpha) - \mu(t_2, \alpha))(y) 
&\le 
|\nabla f|_0 d_1(\mu(t_1, \alpha), \mu(t_2, \alpha)) 
\\
&\le
r |\nabla f|_0 |t_1 - t_2|^{1/2}.
\end{array}
\end{equation}

\subsection{Assumptions}
To proceed, we define a space 
$C^{\delta, 0, m}_{0, 0, m'}$ as the collection of all functions in 
$C^{\delta, 0, m}([0, T]\times [0, 1] \times \mathbb T^d, \mathbb R)$ 
equipped with a 
$C^{0, 0, m'}([0, T]\times [0, 1] \times \mathbb T^d, \mathbb R)$ norm.
For instance, if $f\in C^{0.5, 0, 2}_{0, 0, 2}$, then we write its norm as
$$|f|^{0.5, 0, 2}_{0,0,2} = |f|_{0,0,2} 
= |f|_0 + \sum_i |\partial_{x_i} f|_0 + \sum_{ij} |\partial_{x_ix_j} f|_0.
$$
For more details, we refer to Section \ref{s:appendix}.
\begin{assumption}\label{a:gmfg1}
$b: [0, T]\times [0, 1]\times \mathbb T^d \mapsto  \mathbb R^d$, 
$g: [0,1]^2 \mapsto \mathbb R$, and 
$m_0: [0,1]\times \mathbb T^d\mapsto  \mathbb R^d$ are infinitely smooth functions in all variables.
\end{assumption}
We pose the following assumptions for the cost function $\ell_1$. 
Throughout the paper, since $g$ will be a priori given function, we will suppress $g$ by writing 
$$\ell_1 (\mu, g, t, \alpha, x) = \ell_1(\mu, t, \alpha, x)$$
if this does not cause any confusion.
For convenience, we will write
$$\ell_1[\mu](t, \alpha, x) = \ell_1(\mu, t, \alpha, x) = \ell_1 (\mu, g, t, \alpha, x).$$

\begin{assumption}\label{a:gmfg2}
The mapping
$\mu \mapsto \ell_1[\mu]$ is a  bounded and Lipschitz 
continuous mapping from 
$S^{1/2}$ to 
$C^{0.5, 0, 2}_{0, 0, 1}$, that is, for any $\mu\in S^{1/2}$, $\ell_1[\mu]$ belongs to $C^{0.5, 0, 2}$ and 
$$|\ell_1[\mu]|_{0, 0,1} < M, \quad
|\ell_1[\mu_1] - \ell_1[\mu_2]|_{0, 0, 1} 
\le M\rho(\mu_1, \mu_2),$$
for some $M>0$ independent to the choice of $\mu$.
\end{assumption}

We check that the assumptions are  valid for a class of examples given in Lemma \ref{l:ex01}.
\begin{lemma}
\label{l:ex01}
Suppose
$\ell_2\in C^{\infty}({\mathbb T^d} \times \mathbb T^d, \mathbb R)$ and 
$g$ are given smooth enough. 
Then, the cost $\ell_1$ of \eqref{eq:cost1} satisfies Assumption \ref{a:gmfg2}.
\end{lemma}
\begin{proof}
Let $d = 1$ for the simplicity.
For $\mu \in S^{1/2}$, we have
$$|\ell_1[\mu]|_0 \le |\ell_2|_0 |g|_0,$$
$$|\partial_x \ell_1[\mu]|_0 \le |\partial_x \ell_2|_0 |g|_0,$$
$$|\partial_{xx} \ell_1[\mu]|_0 \le |\partial_{xx} \ell_2|_0 |g|_0,$$
$$\begin{array}
{ll}
|\ell_1(\mu, t_1, \alpha, x) - \ell_1(\mu, t_2, \alpha, x)|
& \displaystyle
\le \int_0^1 \int_{\mathbb T^d} 
|\ell_2(x, y) (\mu(t_1, \alpha', d y) - \mu(t_2, \alpha', d y)) g(\alpha, \alpha') |
d\alpha'
\\ & \displaystyle
\le \int_0^1 |\partial_y \ell_2|_0 
d_1(\mu(t_1, \alpha'), \mu(t_2, \alpha')) g(\alpha, \alpha') d\alpha'
\\ & \displaystyle
\le  |\partial_y \ell_2|_0 |g|_0 |\mu|_{1/2} |t_1 - t_2|^{1/2}.
\end{array}
$$
This implies that $\ell_1[\mu] \in C^{1/2, 0, 2}$ with estimation 
\begin{equation}
\label{eq:estimation11}
|\ell_1[\mu]|_{1/2, 0, 2} \le |\ell_2|_{2,0} |g|_0 (1+ |\mu|_{1/2}).
\end{equation}
Note that \eqref{eq:estimation11} does not give a uniform upper bound 
due to the $\mu$-dependence on the right hand side of the inequality.
Nevertheless, we have a uniform upper bound for the weaker norm $|\cdot|_{0,0,1}$:
$$|\ell_1[\mu]|_{0, 0, 1}  
\le  |\ell_2|_{1,0} |g|_0, 
\ \forall \mu \in S^{1/2}.$$
For $\mu_1, \mu_2 \in S^{1/2}$, we have
$$
\begin{array}
{ll}
\ell_1(\mu_1, t, \alpha, x) - \ell_1(\mu_2, t, \alpha, x)
& \displaystyle
= \int_0^1 \int_{\mathbb T^d} \ell_2(x, y) (\mu_1(t, \alpha', d y) - \mu_2(t, \alpha', d y)) 
g(\alpha, \alpha') d\alpha'
\\ & \displaystyle
\le |\partial_y \ell_2|_0 d_1(\mu_1(t, \alpha), \mu_2(t, \alpha)) |g|_0.
\end{array}
$$
This implies that
$$|\ell_1[\mu_1] - \ell_1[\mu_2]|_0 \le |\partial_y \ell_2|_0 |g|_0  \rho(\mu_1, \mu_2).$$
Similarly, we obtain
$$|\partial_x\ell_1[\mu_1] - \partial_x \ell_1[\mu_2]|_0 \le |\partial_y \partial_x\ell_2|_0 |g|_0  \rho(\mu_1, \mu_2).$$
Therefore, we have Lipschitz continuity
$$|\ell_1[\mu_1] - \ell_1[\mu_2]|_{0, 0, 1}  \le | \ell_2|_{1,1} |g|_0 \rho(\mu_1, \mu_2),$$
and this implies Assumption \ref{a:gmfg2} with $M = |\ell_2|_{1,1} |g|_0.$
\end{proof}

\subsection{Operator $\Phi_1$}
Recall that  $ \nabla  \tilde v = \Phi_1(\mu)$, where $\tilde v$
is the solution of \eqref{eq:gmfg02a} with given $\mu$.
By Hopf-Cole transform $\tilde v$ is the solution of \eqref{eq:gmfg02a} if and only if
\begin{equation}
\label{eq:w01}
w = \exp\{-  \tilde v\}
\end{equation}
is the solution of 
\begin{equation}\label{eq:hjb04}
\left\{
\begin{array}
{ll}
\partial_t w + \frac 1 2 \Delta w 
- w \tilde \ell_1[\mu] = 0
&\hbox{ on }  (0, T)\times [0,1]\times {\mathbb T^d} \\
w(T, \alpha, x) = e^{b(T, \alpha, x)} &\hbox{ on }  [0,1]\times {\mathbb T^d}.
\end{array}
\right.
\end{equation}
In addition, we have the following relation by chain rule:
$$\nabla \tilde v = - \frac{ \nabla w}{w}, \ 
\Delta \tilde v =   \frac{ - w \Delta w + |\nabla w|^2}{w^2}.$$
Since $w$-term appears in the denominator,  
Harnack type inequality in Corollary \ref{c:har01} 
ensures that $\nabla \tilde v$ and $\Delta \tilde v$ are well defined.
\subsubsection{Estimates of parameterized PDEs}
We define 
\begin{equation}
\label{eq:G01}
w = G(f)
\end{equation}
by the solution of
\begin{equation}\label{eq:hjb05}
\left\{
\begin{array}
{ll}
\partial_t w + \frac 1 2 \Delta w 
- w f = 0 &\hbox{ on }   (0, T)\times [0,1]\times {\mathbb T^d} \\
w(T, \alpha, x) = e^{b(T, \alpha, x)}  &\hbox{ on }  [0,1]\times {\mathbb T^d}.
\end{array}
\right.
\end{equation}
Note that $w = G(\tilde \ell_1[\mu])$  is the solution of \eqref{eq:hjb04}.
\begin{lemma}
\label{l:ppde05}
The mapping
$G$ is a locally Lipschitz continuous mapping from 
$C^{0.5, 0, 2}_{0, 0, 1}$ to $C^{1, 0, 2}_{0, 0, 1}$.
\end{lemma}
\begin{proof}
Let $f\in C^{0.5, 0, 2}$ and $w = G(f)$. 
By Theorem \ref{t:ppde01}, we have $w(\alpha) \in C^{1,3}$.
If $\alpha\to \alpha_0$, then $f(\alpha) \to f(\alpha_0)$ holds pointwisely. Together with Dominated Convergence Theorem on the probabilistic representation of $w$, one can conclude $w(t, \alpha, x) \to w(t, \alpha_0, x)$ whenever $\alpha\to \alpha_0$.
Therefore, $w$ belongs to $C^{1, 0, 3}$.

Given $f_1, f_2 \in C^{0.5, 0, 2}$ and $w_i = G(f_i)$ with 
$$K(\alpha) = |f_1(\alpha)|_{0,1} + |f_2(\alpha)|_{0,1} + |e^{b(T, \alpha)}|_3,$$
we can use local Lipschitz continuity of Theorem \ref{t:ppde01} to obtain local Lipschitz of $G$,
$$
\begin{array}
{ll}
|w_1 - w_2|_{0, 0, 1} &
= \sup_\alpha |w_1(\alpha) - w_2(\alpha)|_{0,1}
\\ &
\le \sup_\alpha \Psi(K(\alpha))  |f_1(\alpha) - f_2(\alpha)|_{0,1} 
\\ & 
\le \Psi(\sup_\alpha K(\alpha)) |f_1 -f_2|_{0, 0, 1}.
\end{array}
$$
In the above, we used the monotonicity of $\Psi(\cdot)$ to switch $\Psi$ and $\sup$. 
Since $\sup_\alpha K(\alpha) \le \Psi(|f_1|_{0,0,1} + |f_2|_{0,0,1} + |b|_3)$, we can rewrite the above estimations as
$$|w_1 - w_2|_{0, 0, 1} \le 
 \Psi(|f_1|_{0,0,1} + |f_2|_{0,0,1} + |b|_3)  |f_1 -f_2|_{0, 0, 1}.
$$

\end{proof}

\subsubsection{$\Phi_1$ estimate}

\begin{lemma}
\label{l:Phi02}
$\Phi_1$ is a uniformly bounded and 
Lipschitz continuous mapping from 
$(S^{1/2}, \rho)$ to $C^{0}([0, T] \times[0, 1] \times {\mathbb T^d}, \mathbb R^d)$.
\end{lemma}
\begin{proof}
If $\mu \in S^{1/2}$, then 
 $\ell_1[\mu] \in C^{0.5, 0, 2}$ with $|\ell_1[\mu]|_{0, 0, 1} < M$ by Assumption \ref{a:gmfg2}. 
We recall that 
$$\tilde \ell_1  = \ell_1 + (\partial_t b + \frac 1 2 |\nabla b|^2 + \frac 1 2 \Delta b).$$
Due to the smoothness of $b$ and compactness of its domain,  
we still have $\tilde \ell_1[\mu] \in C^{0.5, 0, 2}$ with $|\tilde \ell_1[\mu]|_{0, 0, 1} < \Psi( M)$.
 Together with local Lipschitz continuity  of $ G(\cdot)$ in Lemma \ref{l:ppde05}, it implies uniform boundedness of 
 $w = G(\tilde \ell_1[\mu])$, i.e. 
 $$|w|_{0,0,1} < \Psi(M).$$
Moreover, Corollary \ref{c:har01} says that the reciprocal of $w = G( \tilde \ell_1[\mu])$ is bounded 
in the sense  $|w^{-1}|_0 < \Psi(| \tilde \ell_1[\mu]|_0)$. Therefore, we have
$$|w|_{0,0,1} + |w^{-1}|_{0} < \Psi(M).$$
Next, we can prove that $\Phi_1$ is uniformly bounded in $C^{0}$:
$$
|\Phi_1(\mu)|_0 = |\nabla \tilde v|_0 =  |w^{-1} \nabla w|_0
\le |w^{-1}|_0 |\nabla w|_0 
\le  |w^{-1}|_0 |w|_{0, 0, 1} \le \Psi(M).
$$
Finally, we can show  the global Lipschitz for $\Phi_1$ by the following estimates:
$$
\begin{array}
{ll}
|\Phi_1(\mu_1) - \Phi_1(\mu_2)|_0 
 & 
=  |w_1^{-1} \nabla w_1 - w_2^{-1} \nabla w_2|_0
\\  & \displaystyle
=  |\frac{w_2 \nabla w_1 - w_1 \nabla w_2}{w_1w_2}|_0
\\  & \displaystyle
\le \Psi(M) (|w_2|_0 |\nabla w_1 - \nabla w_2|_0  + 
|\nabla w_2|_0 | w_1 - w_2|_0 )
\\  & \displaystyle
\le \Psi(M) |w_1 - w_2|_{0,0,1}
\\  & \displaystyle
\le \Psi(M) |\tilde \ell_1[\mu_1] - \tilde  \ell_1[\mu_2]|_{0,0,1}
\\  & \displaystyle
\le  \Psi(M) \rho(\mu_1, \mu_2).
\end{array}
$$
In the last two steps, we used Lipschitz continuity obtained by
Lemma \ref{l:ppde05} and Assumption \ref{a:gmfg2}.
\end{proof}

\subsection{Operator $\Phi_2$}
Next, we will show the properties associated to $\Phi_2$ mapping from
$C^{0}([0, T] \times[0, 1] \times {\mathbb T^d}, \mathbb R^d)$ to $S^{1/2}$.
 \begin{lemma}
\label{l:Phi03}
$\Phi_2$ is a 
locally Lipschitz continuous mapping from 
$C^{0}([0, T] \times[0, 1] \times {\mathbb T^d}, \mathbb R^d)$ to $(S^{1/2}, \rho)$. 
Moreover, $|\Phi_2(\bar v)|_{1/2} \le \Psi(|\bar v|_0)$ 
for all $\bar v\in C^{0}([0, T] \times[0, 1] \times {\mathbb T^d}, \mathbb R^d)$  for some monotonically increasing positive function $\Psi$.
\end{lemma}
\begin{proof}

Given $\bar v \in C^{0}([0, T] \times[0, 1] \times {\mathbb T^d}, \mathbb R^d)$ and $\nu = \Phi_2(\bar v)$, 
applying \eqref{eq:nu02} of Theorem \ref{t:fpk01}, it yields that
$$
\begin{array}{ll}
|\nu|_0 = \sup_{t, \alpha} \int_{{\mathbb T^d}} |x| \nu(t, \alpha, dx) 
&\displaystyle 
= \sup_\alpha \sup_t \int_{{\mathbb T^d}} |x| \nu(t, \alpha, dx) 
\\&\displaystyle 
\le \sup_\alpha \Big( \int_{{\mathbb T^d}} |x| m_0(\alpha, dx) + |\bar v(\alpha)|_0 T + \sqrt T\Big)
\\&\displaystyle 
\le \Psi(|\bar v|_0).
\end{array}
$$
Next, we show the following equicontinuity property again by  \eqref{eq:nu01} of Theorem \ref{t:fpk01}:
$$
\begin{array}
{ll}
\sup_{t_1\neq t_2, \alpha} d_1 (\nu(t_1, \alpha), \nu(t_2, \alpha))
&\displaystyle 
\le \sup_\alpha (1 + \sqrt T |\bar v(\alpha)|_0) |t_1 - t_2|^{1/2}
\\&\displaystyle 
\le \Psi(|\bar v|_0) |t_1 - t_2|^{1/2}.
\end{array}
$$
This proves $\nu \in S^{1/2}$ with
$$|\nu|_{1/2} \le \Psi(|\bar v|_0).$$

For the continuity of $\Phi_2$, given $\bar v_1, \bar v_2 \in 
C^{0}([0, T] \times[0, 1] \times {\mathbb T^d}, \mathbb R^d)$, we set $\nu_i = \Phi_2(\bar v_i)$ for $i =1, 2$. 
Then, we use the local Lipschitz continuity in Theorem \ref{t:fpk01} to obtain
local Lipschitz continuity of $\Phi_2$ as follows:
$$
\begin{array}
{ll}
\rho(\nu_1, \nu_2) & \displaystyle 
= \sup_{t, \alpha} d_1 (\nu_1(t, \alpha), \nu_2(t, \alpha)) 
\\ & \displaystyle 
= \sup_\alpha \sup_t d_1 (\nu_1(t, \alpha), \nu_2(t, \alpha)) 
\\ & \displaystyle 
= \sup_\alpha \Psi(|\bar v_1(\alpha)|_0 + |\bar v_2(\alpha)|_0) |\bar v_1(\alpha) - \bar v_2(\alpha)|_0
\\ & \displaystyle 
\le \Psi(|\bar v_1|_0 + |\bar v_2|_0) |\bar v_1 - \bar v_2|_0.
\end{array}
$$

\end{proof}

\subsection{Existence by Schauder's fixed point theorem}
\begin{theorem}
\label{t:existence}
Suppose Assumptions \ref{a:gmfg1} - \ref{a:gmfg2} are valid.
Then there exists a solution of \eqref{eq:gmfg01} in the space
$C^{1, 0, 2} ([0, T]\times [0, 1] \times {\mathbb T^d}, \mathbb R) \times C([0, T]\times [0, 1], \mathcal P_1({\mathbb T^d})) $.
\end{theorem}
\begin{proof}
It is enough to show that $\Phi_2 \circ \Phi_1$ has a fixed point in $S^{1/2}$. 
Recall that $B_{r}$ is a convex closed and compact subset of $S^{1/2}$. 
For simplicity, we denote by $\hat B_r$  the closed ball of radius $r$ in 
$C^{0}([0, T]\times [0,1]\times {\mathbb T^d}, \mathbb R^d)$.
\begin{enumerate}
\item
By Lemma \ref{l:Phi02}, there exists some positive increasing function $\Psi_1$ independent to $r$, such that the mapping
$$\Phi_1: B_{r} \mapsto \hat B_{\Psi_1(M)}$$
is continuous.
\item By Lemma \ref{l:Phi03}, there exists some positive increasing function $\Psi_2$  such that the mapping 
$$\Phi_2: \hat B_{\Psi_1(M)} \mapsto B_{\Psi_2 \circ \Psi_1 (M)}$$
is continuous.
\end{enumerate}
Now we take 
$$ r = \Psi_2 (\Psi_1(M))$$
and we have
$$\Phi_2 \circ \Phi_1: B_{r} \mapsto B_{r}$$
is a continuous mapping and this yields the existence of a fixed point for $\Phi$ by Schauder's theorem.
\end{proof}

In the above, we have indeed proved the existence in the space $C^{1, 0, 2} ([0, T]\times [0, 1] \times {\mathbb T^d}, \mathbb R) \times S^{1/2}$.

\subsection{Further remarks on the fixed point theorem} \label{s:remarks}
In connection with GMFG, we explain why Theorem  \ref{t:ppde01} establishes 
locally Lipschitz continuity of the solution map $u: [c, f, \psi] \mapsto u[c, f, \psi]$  of \eqref{eq:ppde04} 
in the sense of
\begin{equation}
\label{eq:map01}
 C^{\delta,2}_{0,1} \times C^{\delta,2}_{0,1} \times C^{4}_3
\mapsto 
C^{1,3}_{0,1}
\end{equation}
instead of
\begin{equation}
\label{eq:map02}
 C^{\delta,2} \times C^{\delta,2} \times C^{4}
\mapsto 
C^{1,3}.
\end{equation}

For the illustration purpose, if we freeze $c, \psi$ of the solution map $u$, then local Lipschitz continuity in the sense of
\eqref{eq:map01} implies local boundedness
$$|u|_{0,1} \le \Psi(|f|_{0,1}),
$$
while local Lipschitz continuity in the sense of
\eqref{eq:map02} implies local boundedness
$$|u|_{0,1} \le |u|_{1,3} \le \Psi(|f|_{\delta,2}).
$$
The main difference of these two local boundedness properties is that, the first one controls $u$ by $f$ with $0$-norm in $t$-variable while the second one does by $f$ with $\delta$-norm in $t$-variable, which is not desirable. 
The main reason is that
the running cost $|\ell_1[\mu]|_{1/2, 0, 1} \le \Psi(|\mu|_{1/2})$ of \eqref{eq:estimation11} does not have uniform bound in $\mu$, while
$|\ell_1[\mu]|_{0, 0, 1}$ does.
For this reason, we included the regularity results for parabolic PDE solutions 
by dropping $t$-regularity while increasing $x$-regularity as a tradeoff.

Recall that, we have established the existence of a fixed point of a mapping 
$\Phi = \Phi_2 \circ \Phi_1$ for $\Phi_1: \mu \mapsto \nabla  \tilde v$ and $\Phi_2: \nabla  \tilde v \mapsto \nu$.
Our approach is along the the Schuader's fixed point theorem with estimates
$$\Phi_1: B_{r} \mapsto \hat B_{\Psi_1(M)}, 
\ 
\Phi_2: \hat B_{\Psi_1(M)} \mapsto B_{\Psi_2 \circ {\Psi_1(M)}}.
$$
In the above, it is crucial that the $\Phi_1$ is upper bounded by $\Psi_1(M)$ independent to $r$, 
and this can be inferred from local boundedness of \eqref{eq:map01} together with uniform boundedness of 
$|\ell_1[\mu]|_{0,0,1}$. 

In contrast, if we use local boundedness in the sense of \eqref{eq:map02}, then we have estimations in the form of
$$\Phi_1: B_{r} \mapsto \hat B_{\Psi_1(r)}, 
\ 
\Phi_2: \hat B_{\Psi_1(r)} \mapsto B_{\Psi_2 \circ {\Psi_1(r)}}.
$$
Since  the norm of the running cost $|\ell_1[\mu]|_{1,0,3}$  depends on $\mu$, $\Phi_1$ can not be uniformly bounded.
As a result, 
the choice of  $r = \Psi_1(r)$ is infeasible.

\section{Uniqueness of GMFG}\label{s:unique}
\begin{assumption} \label{a:mono}
There exists some $\alpha \in [0,1]$ satisfying 
$$
\int_{\mathbb T^d}  ( \ell_1(\mu_1, g, t, \alpha, x) - \ell_1( \mu_2, g, t,  \alpha, x)) (\mu_1 - \mu_2)(t, \alpha, dx)>0,
$$
for all $\mu_1 \neq \mu_2 \in C([0, T] \times [0, 1], \mathcal P_{1}(\mathbb{T}^{d}))$ and $t\in [0, T]$.
\end{assumption}

\begin{theorem}  (\cite{Car13}, \cite{Ryz18})
\label{t:uniqueness}
Suppose Assumptions \ref{a:gmfg1} - \ref{a:gmfg2} and \ref{a:mono} are valid.
Then, there exists a unique solution of \eqref{eq:gmfg01} in the space
$C^{1, 0, 2} ([0, T]\times [0, 1] \times {\mathbb T^d}, \mathbb R) \times C([0, T]\times [0, 1], \mathcal P_1({\mathbb T^d})) $.
\end{theorem}

\begin{proof}
For $i = 1,2$, let $(v_i, \mu_i)$ be two different solution pairs, and set
$$\bar v = v_1 - v_2, \ \bar \mu = \mu_1 - \mu_2.$$
Note that $\bar v(T, \alpha, x) = \bar \mu(0, \alpha, x) = 0$ for all $(\alpha, x)$ by their given initial-terminal data. 
We also write $\ell_1[\mu_i] = \ell_1[\mu_i, g]$ for short.
Then $\bar v$ satisfies
$$\partial_t \bar v + \nabla b \cdot \nabla \bar v + \frac 1 2 \Delta \bar v 
- \frac 1 2 |\nabla v_1|^2 + \frac 1 2 |\nabla v_2|^2 + \ell_1[\mu_1] - \ell_1[\mu_2] = 0$$
and $\bar \mu$ satisfies
$$- \partial_t \bar \mu - \text{div} (\nabla b \bar \mu)
+ \frac 1 2 \Delta \bar \mu
+ \text{div} (\nabla v_1 \mu_1) - 
\text{div} (\nabla v_2 \mu_2) = 0.$$
The above two equations can be rewritten as
$$\langle
\partial_t \bar v + \nabla b \cdot \nabla \bar v + \frac 1 2 \Delta \bar v, \bar \mu
\rangle
+ 
\langle
- \frac 1 2 |\nabla v_1|^2 + \frac 1 2 |\nabla v_2|^2 + \ell_1[\mu_1] - \ell_1[\mu_2]  , 
\bar \mu \rangle = 0$$
and
$$ \langle
\partial_t \bar v + \nabla b \cdot \nabla \bar v + \frac 1 2 \Delta \bar v, \bar \mu
\rangle
+ \langle 
\bar v, 
\text{div} (\nabla v_1 \mu_1) - 
\text{div} (\nabla v_2 \mu_2)
 \rangle= 0.$$
 By subtracting above two equations, 
and utilizing the identities 
$$
\langle 
\text{div} (\nabla v_1 \mu_1), \bar v
\rangle = 
- \langle |\nabla v_1|^2, \mu_1\rangle + 
\langle \nabla v_1 \cdot \nabla v_2, \mu_1\rangle
$$
and
$$
\langle 
\text{div} (\nabla v_2 \mu_2), \bar v
\rangle = 
\langle |\nabla v_2|^2, \mu_2\rangle -
\langle \nabla v_1 \cdot \nabla v_2, \mu_2
\rangle,
$$
we obtain
$$
\langle \frac 1 2(\mu_1 + \mu_2), |\nabla \bar v|^2\rangle + 
\langle \ell_1[\mu_1] - \ell_1[\mu_2] , 
\bar \mu \rangle = 0.
$$
The first term is non-negative and the second term is strictly positive for some $\alpha \in [0,1]$ by (A3), 
which implies a contradiction.
\end{proof}

\section{Concluding remarks}\label{s:conclusion}
Our main result of Theorem \ref{t:uniqueness} provides existence and uniqueness of the GMFG equation
under some assumptions. 
One limitation of the current setting is that the running cost in the current setup allows to use Hopf-Cole transformation, which is essential to the subsequent analysis on regularities. To deal with the full generalization of the running cost, one must adopt different approaches and it will be  in our future research direction. It is also desirable to generalize the result to the whole domain $\mathbb R^d$ instead of compact domain $\mathbb T^d$.
Another limitation is that, the current setting requires the continuity of the graphon. 
Note that some graphons are not necessarily continuous. 
Nevertheless, the continuity condition of the graphon can be relaxed in the following sense by similar arguments with additional complexity of notations, which is sketched below briefly.

To proceed, we define $\hat C^{0}$ as the collection of bounded 
measurable functions 
$f: [0, T]\times [0,1] \times \mathbb T^d \mapsto \mathbb R$, and 
we denote its norm as
$$|f|_{0} = \sup_{[0, T]\times [0,1] \times \mathbb T^d} |f(t, \alpha, x)|.$$
With $\hat C^{\delta, 0, 2}$, we denote the set of functions $f \in \hat C^0$ with finite norm
$$|f|_{\delta,0,2} = |f|_0 + \sup_{t_1<t_2, \alpha, x} \frac{|f(t_1, \alpha, x) - f(t_2, \alpha, x)|}{|t_1 -t_2|^\delta} + \sum_{i} |\partial_i f|_0 + \sum_{ij} |\partial_{ij} f|_0.$$
By the above definition $\hat C^{\delta, 0, 2}$ allows the discontinuity in $\alpha$.
\begin{assumption}\label{a:gmfg4}
\begin{enumerate}
\item $b$ and $m_0$ are infinitely smooth in their domains.
\item 
The graphon $g$ is  bounded measurable  on $[0,1]^2$ with $$|g|_0 = \sup_{[0,1]^2} |g(\alpha, \alpha')| < \infty.$$
\end{enumerate}
\end{assumption}

We recall that $B_{r}$ is defined in $S^{1/2}$. 
We use $\hat C^{\delta, 0, 2}_{0, 0, 2}$ to denote the same set $\hat C^{\delta, 0, 2}$ with the norm $|\cdot|_{0,0,2}$, i.e.
$$|f|_{0,0,2} = |f|_0  + \sum_{i} |\partial_i f|_0 + \sum_{ij} |\partial_{ij} f|_0.$$

\begin{assumption}\label{a:gmfg5}
The mapping
$\mu \mapsto \ell_1[\mu]$ is a  bounded and Lipschitz 
continuous mapping from 
$S^{1/2}$ to 
$\hat C^{0.5, 0, 2}_{0, 0, 1}$, that is, for any $\mu\in S^{1/2}$, $\ell_1[\mu]$ belongs to $\hat C^{0.5, 0, 2}$ and 
$$|\ell_1[\mu]|_{0, 0,1} < M, \quad
|\ell_1[\mu_1] - \ell_1[\mu_2]|_{0, 0, 1} 
\le M\rho(\mu_1, \mu_2),$$
for some $M>0$ independent to the choice of $\mu$.
\end{assumption}

We also define $\hat C^{m, 0, n}$ as the collection of $f \in \hat C^0$ with continuous bounded $m$-th derivatives in $t$ and $n$-th derivatives $x$. For instance, for $f\in \hat C^{1,0,2}$, we have finite norm
$$|f|_{1,0,2} = |f|_0 + |\partial_t f|_0 + \sum_{i} |\partial_i f|_0 + \sum_{ij} |\partial_{ij} f|_0.$$
Now we present a result in parallel to Theorem \ref{t:uniqueness}.
The proof is similar and so omitted.
\begin{corollary}
\label{c:uniqueness}
Suppose Assumptions \ref{a:gmfg4} - \ref{a:gmfg5} and \ref{a:mono} are valid.
Then there exists a unique solution of \eqref{eq:gmfg01} in the space
$\hat C^{1, 0, 2} ([0, T]\times [0, 1] \times {\mathbb T^d}, \mathbb R) 
\times \hat C([0, T]\times [0, 1], \mathcal P_1({\mathbb T^d})) $.
\end{corollary}

\section{Appendix}\label{s:appendix}

In this appendix, we will summarize the notations of H\"older space used in this paper. 
For this purpose, we will define the following functionals for a function $u$ from a product normed space $S = X\times Y$ to $\mathbb R^d$ whenever it is well defined.
\begin{itemize}
\item
$|u|_0 = \sup_S |u(x,y)|$.
\item For nonnegative integers $l, m$, define
$$|u|_{l, m} = 
\sum_{i=0}^l  \sum_{|\alpha| = i} |D_x^\alpha  u|_0
+ \sum_{i=0}^m  \sum_{|\alpha| = i} |D_y^\alpha u|_0.$$
In the above, $\alpha$ is a multi-index for differential operators. For instance, $|\alpha| = \sum_{i= 1}^{d_1} |\alpha_i|$ for a multi-index 
$\alpha = (\alpha_i : i = 1, \dots, d_1)$.

\item For positive numbers $l', m' \in (0,1)$, define
$$[u]_{l', m'} = [u]_{l', 0} + [u]_{0, m'},$$ where
$$
[u]_{l', 0} = \sup_{x_1 \neq x_2, y} \frac{|u(x_1, y) - u(x_2, y)|}
{|x_1 - x_2|^{l'}}, 
$$
and
$$
[u]_{0, m'} = \sup_{x, y_1 \neq y_2} \frac{|u(x, y_1) - u(x, y_2)|}
{|y_1 - y_2|^{m'}}.
$$

\item For nonnegative integers $l, m$ and positive number $l' \in (0,1)$, define
$$|u|_{l+ l', m} = |u|_{l, m} + 
\sum_{|\alpha| = l} [D_x^\alpha  u]_{l', 0} .$$
\item For nonnegative integers $l, m$ and positive numbers $l', m' \in (0,1)$, define
$$|u|_{l+ l', m+ m'} = |u|_{l, m} + 
\sum_{|\alpha| = l} [D_x^\alpha  u]_{l', m'}
+  \sum_{|\alpha| = m} [D_y^\alpha u]_{l', m'}.$$
\end{itemize}
One can check that the following spaces are Banach spaces:
\begin{itemize}
\item
$C^{l, m} (X\times Y; \mathbb R^d) := 
\{u: |u|_{l, m}<\infty\},$
\item 
$C^{l+ l', m}  (X\times Y; \mathbb R^d)  := \{u:  |u|_{l+ l', m} <\infty\},$
\item 
$C^{l+ l', m+ m'}  (X\times Y; \mathbb R^d)  := \{u:  |u|_{l+ l', m+ m'} <\infty\}.$
\end{itemize}
In this paper, we also involve the space $C^{l', 0, m}$ of functions with a domain  $S = X\times Y\times Z$, whose norm is defined as
$$|u|_{l', 0, m} = |u|_{0, 0, m} + [D_z^m u]_{l', 0, 0}, $$
where
$$|u|_{0, 0, m} = \sum_{i=0}^m \sum_{|\alpha| = i} |D_z^\alpha u|_0, \hbox{ and }
[u]_{l',0,0} = \sup_{x_1 \neq x_2, y, z} \frac{|u(x_1, y, z) - u(x_2, y, z)|}
{|x_1 - x_2|^{l'}}.
$$

In this paper, our functions involve state domain taking values in 
$d$-torus $\mathbb T^d = \mathbb R^d/\mathbb Z^d$.
For $x\in \mathbb R^d$, let $\pi(x)$ be the coset of $\mathbb Z^d$ that contains $x$, i.e. 
$$\pi(x) = x + \mathbb Z^d.$$
A canonical metric on $\mathbb T^d$ can be induced from the Euclidean metric by
$$|\pi(x) - \pi(y)|_{\mathbb T^d} = \inf\{|x - y - z|: z\in \mathbb Z^d\}.$$
For the illustration purpose, we provide  a list of representative H\"older spaces used throughout the paper:
\begin{itemize}
\item
$C^{\delta/2, \delta}([0, T]\times \mathbb T^d)$ is a space of functions $u(t, x)$ with a norm defined by
$$|u|_{\delta/2, \delta} = |u|_0 + [u]_{\delta/2, \delta},$$
where $[u]_{\delta/2, \delta}$ is a seminorm defined by
$$[u]_{\delta/2, \delta} = 
 \sup_{t_1 \neq t_2, x} \frac{|u(t_1, x) - u(t_2, x)|}{|t_1 - t_2|^{\delta/2}}
 + 
 \sup_{t, x_1 \neq x_2} \frac{|u(t, x_1) - u(t, x_2)|}
{|x_1 - x_2|^{\delta}}.$$
This definition may be slightly different from different resources. 
For instance, the definition given by \cite{Kry96} for the seminorm is 
$$
[u]'_{\delta/2, \delta} = \sup_{(t_1, x_1)\neq (t_2, x_2) \in [0, T]\times \mathbb T^d} 
\frac{|u(t_1, x_1) - u(t_2, x_2)|}
{(|t_1-t_2|^{1/2} + |x_1 - x_2|)^{\delta}}.
$$
Indeed, two norms induced by $[u]_{\delta/2, \delta}$ and $[u]'_{\delta/2, \delta}$ are equivalent, 
which can be seen from
below:
$$[u]_{\delta/2, \delta} = [u]_{\delta/2, 0} + [u]_{0, \delta} \le 2 [u]'_{\delta/2, \delta},$$
and
$$
\begin{array}
{ll}
[u]'_{\delta/2, \delta} & \le \sup_{(t_1, x_1)\neq (t_2, x_2)} 
\frac{|u(t_1, x_1) - u(t_2, x_1)| + |u(t_2, x_1) - u(t_2, x_2)|}
{(|t_1-t_2|^{1/2} + |x_1 - x_2|)^{\delta}} 
\\
&\le  \sup_{t_1\neq t_2} 
\frac{|u(t_1, x_1) - u(t_2, x_1)|}
{|t_1-t_2|^{\delta/2} } + 
 \sup_{ x_1 \neq  x_2} 
\frac{|u(t_2, x_1) - u(t_2, x_2)|}
{ |x_1 - x_2|^{\delta}} \\
& \le [u]_{\delta/2, 0} + [u]_{0, \delta} = [u]_{\delta/2, \delta}.
\end{array}
$$

\item $C^{0,1}([0, T]\times \mathbb T^d)$ is a space of functions $u(t, x)$ with a norm
$$|u|_{0,1} = |u|_0 + 
\sum_{i=1...d} |\partial_{x_i} u|_0,$$
and 
$C^{\delta, 1}([0, T]\times \mathbb T^d)$ is a space of functions $u(t, x)$ with a norm
$$|u|_{\delta,1} = |u|_{0,1} + \sum_{i=1...d} [\partial_{x_i} u]_{\delta, 0}.$$

\item 
$C^{1, 2}([0, T]\times \mathbb T^d)$ is a space of functions $u(t, x)$ with a norm
$$|u|_{1,2} = |u|_0 + |\partial_t u|_0 + 
\sum_{i=1...d} |\partial_{x_i} u|_0 +
\sum_{i,j=1...d} |\partial_{x_i x_j} u|_0.$$
\item 
$C^{1+\delta/2, 2+ \delta}([0, T]\times \mathbb T^d)$ is a space with a norm
$$|u|_{1+\delta/2, 2+ \delta} = |u|_{1,2} + [\partial_t u]_{\delta/2, \delta} + 
\sum_{i,j=1...d} [\partial_{x_i x_j} u]_{\delta/2, \delta}.$$
\item 
$C^{0, 2}([0, T]\times \mathbb T^d)$ is a space with a norm
$$|u|_{0,2} = |u|_0 +  
\sum_{i=1...d} |\partial_{x_i} u|_0 +
\sum_{i,j=1...d} |\partial_{x_i x_j} u|_0.$$
$C^{\delta, 2}([0, T]\times \mathbb T^d)$ is a space with a norm
$$|u|_{\delta, 2} = |u|_{0,2} +
\sum_{i,j=1...d} [\partial_{x_i x_j} u]_{\delta, 0}.$$ 
We use $C^{\delta, 2}_{0,2}([0, T]\times \mathbb T^d)$ to denote the space of all functions in $C^{\delta, 2}([0, T]\times \mathbb T^d)$ topologized by the norm $|\cdot|_{0,2}$. 
Such a space $C^{\delta, 2}_{0,2}([0, T]\times \mathbb T^d)$ is not complete. 
However, every $|\cdot|_{\delta, 2}$-norm bounded ball in $C^{\delta, 2}_{0,2}([0, T]\times \mathbb T^d)$ is precompact since 
$C^{\delta, 2}([0, T]\times \mathbb T^d)$ is compactly embedded into $C^{0, 2}([0, T]\times \mathbb T^d)$.

\item $C^{1,3}_{0,1}([0, T]\times \mathbb T^d)$ is the space of all $u\in C^{1,3}([0, T]\times \mathbb T^d)$ topologized by $|\cdot|_{0,1}$, i.e.
$$|u|_{0,1} = |u|_0 + 
\sum_{i=1...d} |\partial_{x_i} u|_0 .$$

\item $C^{0,0,2}([0, T]\times [0,1] \times \mathbb T^d)$ is the space of all $u(t, \alpha, x)$ having finite norm of
$$
|u|_{0,0,2} =  |u|_0 +  
\sum_{i=1...d} |\partial_{x_i} u|_0 +
\sum_{i,j=1...d} |\partial_{x_i x_j} u|_0.
$$
\item $C^{\delta,0,2}([0, T]\times [0,1] \times \mathbb T^d)$ is the space of all $u(t, \alpha, x)$ having finite norm of
$$
|u|_{\delta,0,2} = |u|_{0,0,2} +  \sum_{i,j = 1...d} [\partial_{x_i x_j} u]_{\delta, 0, 0}.
$$
\item 
$C^{\delta,0,2}_{0,0,2}([0, T]\times [0,1] \times \mathbb T^d)$ is the space of all $u(t, \alpha, x)$ having finite norm of $|u|_{\delta,0,2}$, but topologized by $|\cdot|_{0,0,2}$.

\end{itemize}




\bibliographystyle{plain}


\begin{thebibliography}{10}

\bibitem{CM18}
Peter~E. Caines and Minyi Huang.
\newblock Graphon mean field games and the gmfg equations.
\newblock In {\em 2018 IEEE Conference on Decision and Control (CDC)}, pages
  4129--4134. IEEE, 2018.

\bibitem{CH19}
Peter~E. Caines and Minyi Huang.
\newblock Graphon mean field games and the gmfg equations: $\varepsilon$-nash
  equilibria.
\newblock In {\em 2019 IEEE 58th Conference on Decision and Control (CDC)},
  pages 286--292. IEEE, 2019.
  
\bibitem{CH20}
Peter~E. Caines and Minyi Huang.
\newblock Graphon Mean Field Games and Their
Equations.
\newblock SIAM Journal on Control and Optimization. 2021;59(6):4373-
99; arXiv:2008.10216v1

\bibitem{Car13}
P~Cardaliaguet.
\newblock Notes on mean field games.
\newblock {
  \url{https://www.ceremade.dauphine.fr/~cardaliaguet/MFG20130420.pdf}
  }, 2013.

\bibitem{Eva98}
Lawrence~C. Evans.
\newblock {\em Partial differential equations}, volume~19 of {\em Graduate
  Studies in Mathematics}.
\newblock American Mathematical Society, Providence, RI, 1998.

\bibitem{FS06}
W.~H. Fleming and H.~M. Soner.
\newblock {\em Controlled {M}arkov processes and viscosity solutions},
  volume~25 of {\em Stochastic Modelling and Applied Probability}.
\newblock Springer, New York, second edition, 2006.

\bibitem{GC18}
Shuang Gao and Peter~E Caines.
\newblock Grapbon linear quadratic regulation of large-scale networks of linear
  systems.
\newblock In {\em 2018 IEEE Conference on Decision and Control (CDC)}, pages
  5892--5897. IEEE, 2018.

\bibitem{GPV16}
D.A. Gomes, E.A. Pimentel, and V.~Voskanyan.
\newblock {\em Regularity Theory for Mean-Field Game Systems}.
\newblock SpringerBriefs in Mathematics. Springer International Publishing,
  2016.

\bibitem{HCM03}
M.~Huang, P.~E. Caines, and R.~P. Malhame.
\newblock Individual and mass behavior in large population stochastic wireless
  power control problems: centralized and nash equilibrium solutions.
\newblock {\em In: Proceedings of the 42nd IEEE CDC}, pages 98--103, 2003.

\bibitem{HCM06}
M.~Huang, P.~E. Caines, and R.~P. Malhame.
\newblock Large population stochastic dynamic games: closed- loop mckean-vlasov
  systems and the nash certainty equivalence principle.
\newblock {\em Commun Inf Syst}, 6(3):221--251, 2006.

\bibitem{HCM07}
M.~Huang, P.~E. Caines, and R.~P. Malhame.
\newblock Large-population cost-coupled lqg problems with non-uniform agents:
  individual-mass behavior and decentralized epsilon-nash equilibria.
\newblock {\em IEEE Trans Automatic Control}, 52:1560--1571, 2007.

\bibitem{Kel17}
John~L Kelley.
\newblock {\em General topology}.
\newblock Courier Dover Publications, 2017.

\bibitem{Kry96}
N.~V. Krylov.
\newblock {\em Lectures on elliptic and parabolic equations in {H}\"older
  spaces}, volume~12 of {\em Graduate Studies in Mathematics}.
\newblock American Mathematical Society, Providence, RI, 1996.

\bibitem{LSU67}
O.~A. Lady{\v{z}}enskaja, V.~A. Solonnikov, and N.~N. Ural{\cprime}ceva.
\newblock {\em Linear and quasilinear equations of parabolic type}.
\newblock Translations of Mathematical Monographs, Vol. 23. American
  Mathematical Society, Providence, R.I., 1967.

\bibitem{LL07}
J-M Lasry and P-L Lions.
\newblock Mean field games.
\newblock {\em Japan J. Math}, 2:229--260, 2007.

\bibitem{Lov12}
L{\'a}szl{\'o} Lov{\'a}sz.
\newblock {\em Large networks and graph limits}, volume~60.
\newblock American Mathematical Soc., 2012.

\bibitem{NC13}
M.~Nourian and P.~E. Caines.
\newblock {$\epsilon$}-{N}ash mean field game theory for nonlinear stochastic
  dynamical systems with major and minor agents.
\newblock {\em SIAM J. Control Optim.}, 51(4):3302--3331, 2013.

\bibitem{Ryz18}
L~Ryzhik.
\newblock Notes on mean field games.
\newblock {
  \url{https://math.stanford.edu/~ryzhik/STANFORD/MEAN-FIELD-GAMES/notes-mean-field.pdf}
  },
  2018.

\end{thebibliography}

\def\cprime{$'$}

\end{document}